\documentclass[12pt]{article}

\usepackage[margin=1in]{geometry}
\usepackage[usenames]{xcolor}
\usepackage{amssymb}
\usepackage{relsize}
\usepackage{amsmath}
\usepackage{amsthm}
\usepackage{amsfonts}
\usepackage{amscd}
\usepackage{graphicx}
\usepackage{hyperref}
\usepackage{thmtools}
\usepackage{thm-restate}
\usepackage{tikz-cd} 
\usepackage{float}

\begin{document}

\theoremstyle{plain}
\newtheorem{theorem}{Theorem}
\newtheorem{corollary}[theorem]{Corollary}
\newtheorem{lemma}[theorem]{Lemma}
\newtheorem{proposition}[theorem]{Proposition}

\theoremstyle{definition}
\newtheorem{defn}{Definition}
\newtheorem{example}[theorem]{Example}
\newtheorem{conjecture}[theorem]{Conjecture}
\newtheorem{question}[theorem]{Question}

\theoremstyle{remark}
\newtheorem{remark}[theorem]{Remark}

\begin{center}
\vskip 1cm{\LARGE\bf 
Reciprocity Relations for Summations of Squares of Floor Functions and Fractional Parts of Fractions  \\ 
\vskip .1in

}
\vskip 1cm
\large
Damanvir Singh Binner \\
Department of Mathematics\\
Simon Fraser University \\
Burnaby, BC V5A 1S6\\
Canada \\
 dbinner@sfu.ca
\end{center}

\vskip .2in

\begin{abstract}
Given positive coprime integers $a$ and $b$ and a natural number $h$, we obtain reciprocity relations which can be used to quickly evaluate summations like $\sum_{i=1}^{h} \{\frac{ib}{a}\}^2$ and $\sum_{i=1}^{h} \lfloor \frac{ib}{a} \rfloor^2$, where $\lfloor x \rfloor$ and $\{x\}$ denote the floor function and the fractional part of $x$, respectively.
 \end{abstract}
 
\section{Introduction}
\label{Sec1}

We introduce the following notation.

\begin{itemize}
\item $T_1(a,b;h):= \sum_{i=1}^{h} \{\frac{ib}{a}\}^2$.
\item $T_2(a,b;h):= \sum_{i=1}^{h} i \lfloor \frac{ib}{a} \rfloor$.
\item $T_3(a,b;h):= \sum_{i=1}^{h} \lfloor \frac{ib}{a} \rfloor^2$.
\end{itemize}

We can reformulate these sums as follows. Let $q_i$ and $r_i$ denote the quotient and remainder when $ib$ is divided by $a$. Then 
\begin{align*}
T_1(a,b;h) &= \frac{1}{a^2} \sum_{i=1}^h r_i^2, \\
T_2(a,b;h) &= \sum_{i=1}^{h} iq_i, \\
T_3(a,b;h) &= \sum_{i=1}^{h} q_i^2.
\end{align*} 

Note that summations like $\sum_{i=1}^{h} ir_i $ and $\sum_{i=1}^{h} q_i r_i$ can be easily expressed in terms of these sums using the division algorithm. We remark in passing that in $2020$, the present author described a reciprocity relation which can be used to quickly calculate $\sum_{i=1}^{h} q_i $ and $\sum_{i=1}^{h} r_i$ (see \cite[Lemma 7]{Binner}). This reciprocity relationship is also described in Theorem \ref{Quotient} below.

In Section \ref{Sec2}, we derive reciprocity relations for $T_1(a,b;h)$. Using these, we then obtain a reciprocity relation for $T_2(a,b;h)$ in Section \ref{Sec3}. These reciprocity relations help us to easily calculate $T_1(a,b;h)$ and $T_2(a,b;h)$. In Section \ref{Sec4}, we show that $T_1(a,b;h)$ and $T_2(a,b;h)$ can be calculated in $O(\log t)$ and $O((\log t)^2)$ steps, where $t=\max(a,b)$ and by a step, we mean a basic arithmetic operation on the bits of $a$ and $b$. Further we show that $T_3(a,b;h)$ can be easily calculated using the values of $T_1(a,b;h)$ and $T_2(a,b;h)$.
In Sections \ref{Sub1} and \ref{Sub2}, we demonstrate our formulas for an example. Let $q_i$ and $r_i$ denote the quotients and remainders when $2732 \hspace{.05cm} i$ is divided by $8411$. By performing only a few steps, we show that
\begin{align*}
 \sum_{i=1}^{1221} r_i^2 &= 28850219593, \\
 \sum_{i=1}^{1221} iq_i &= 196956430, \\
 \sum_{i=1}^{1221} q_i^2 &= 63853169.
\end{align*} 

We require three main results. The first one is the following well-known result of Sylvester.

\begin{theorem}[Sylvester (1882)]
\label{Sylvester}
 If $a$ and $b$ are positive coprime numbers, the number of natural numbers 
 that cannot be expressed in the form $ax + by$  for nonnegative integers $x$ and $y$ is equal to $\frac{(a-1)(b-1)}{2}$. 
\end{theorem}

This result can be found in \cite{Sylvester82}. Moreover, Sylvester posed this as a recreational problem, and Curran \cite{Sylvester} published a short proof based on generating functions.

Let $NR(a,b)$ denotes the set of nonnegative integers nonrepresentable in terms of $a$ and $b$. That is, $NR(a,b)$ is the set of nonnegative integers $n$ that cannot be expressed in the form $ax + by$. Then, by Theorem \ref{Sylvester}, $|NR(a,b)| = \frac{(a-1)(b-1)}{2}$. In $1993$, Brown and Shiue \cite{BS} discovered the sum $S(a,b)$ of natural numbers that cannot be expressed in the form $ax+by$. 

\begin{theorem}[Brown and Shiue (1993)]
\label{Brown}
 For positive coprime numbers $a$ and $b$, $$ S(a,b) := \sum_{n \in NR(a,b)} n = \frac{1}{12}(a-1)(b-1)(2ab-a-b-1). $$
\end{theorem}
For various calculations involved in our examples, we need the following reciprocity relationship proved by the present author in $2020$.
\begin{theorem}[Binner(2020)]
\label{Quotient}
 Let $a$, $b$, $d$, and $K$ be positive integers such that $b < a$, $d < a$, $\gcd(a,b)  = 1$, and $K = \left\lfloor \frac{bd}{a} \right\rfloor$. Then  $$\sum_{i=1}^{d} \left\lfloor \frac{ib}{a} \right\rfloor + \sum_{i=1}^{K} \left\lfloor \frac{ia}{b} \right\rfloor = dK.$$ 
\end{theorem}

\section{An algorithm for $T_1(a,b;h)$}
\label{Sec2}

In this section, we derive reciprocity relations which can be used to calculate $T_1(a,b;h)$.

\subsection{Reciprocity relation}

Define $$ S(a,b;h) := \left(\frac{a}{2}\right) T_1(a,b;h) +  \left(\frac{a}{2} + 1\right) \sum_{i=1}^{h} \left \lfloor \frac{ib}{a} \right \rfloor. $$  We describe reciprocity relations for $S(a,b;h)$. This leads to a method to quickly calculate $T_1(a,b;h)$ because  $ \sum_{i=1}^{h} \left \lfloor \frac{ib}{a} \right \rfloor$ can be easily calculated using the algorithm described in \cite[Section 2.3]{Binner}. We need some more notation.

\begin{itemize}
\item $n_0$ is the remainder obtained upon dividing $-b(h+1)$ by $a$. 
\item $n:=ab-a+n_0$. Note that $ab-a \leq n < ab$.
\item $H:= n_1-1$, where $n_1$ is the remainder when $-na^{-1}$ is divided by $b$.
\end{itemize}

Our approach is to calculate the number of nonnegative integer solutions $(x,y,z,u)$ of the equation $ax+by+z+u=n$ in two different ways. First, we use Theorems \ref{Sylvester} and \ref{Brown} to find the number of solutions of this equation.

\begin{lemma}
\label{Ref1}
The number of nonnegative integer solutions of the equation $ax+by+z+u=n$ is given by $$ \frac{(n+1)(n+2)}{2} + \frac{(a-1)(b-1)}{12} (2ab-a-b-6n-7). $$
\end{lemma}

\begin{proof}
It is well-known that the equation $ax+by=n$ has either $0$ or $1$ solutions if $n<ab$ (see \cite[Lemma 2 and Lemma 4]{AT}). We view the equation $ax+by+z+u=n$ as the pair of equations $ax+by=i$ and $z+u=n-i$, as $i$ varies from $0$ to $n$. Note that the former equation has a solution only if $i \not\in NR(a,b)$. Then the required number of solutions of the equation $ax+by+z+u=n$ is given by $$  \sum_{\substack{i=0, \\ i \not\in NR(a,b)}}^n (n+1-i). $$ 
Using Theorems \ref{Sylvester} and \ref{Brown} and simplifying, we get that

\begin{align*}
& \sum_{\substack{i=0, \\ i \not\in NR(a,b)}}^n (n+1-i) \\
&=\sum_{i=0}^n (n+1-i) - \sum_{i \in NR(a,b)} (n+1-i) \\
&=  \frac{(n+1)(n+2)}{2} -(n+1)|NR(a,b)| + \sum_{i \in NR(a,b)} i \\
&= \frac{(n+1)(n+2)}{2} -(n+1)\frac{(a-1)(b-1)}{2} + \frac{1}{12}(a-1)(b-1)(2ab-a-b-1) \\
&=  \frac{(n+1)(n+2)}{2} + \frac{(a-1)(b-1)}{12} (2ab-a-b-6n-7).
 \end{align*}
 
\end{proof}

Next, we find the number of solutions of this equation using the method of generating functions. Though our method is similar in spirit to the proof of \cite[Theorem 5]{Binner}, there are several key differences and we provide all the details here for the sake of completeness. We require some more notation. 
\begin{align*}
&\alpha(a,b):= \frac{ab(a+b-2)}{2}, \\
&\beta(a,b):= \frac{ab(a-1)(b-1)}{2} + \frac{ab \left((a-1)(a-2)+(b-1)(b-2) \right)}{3}, \\
&\gamma(a,b):= \frac{2\alpha^2(a,b)-ab\beta(a,b)}{2(ab)^3}, \\
& \eta_1(a,b,h):= (h+H+1) + n\gamma(a,b) + \frac{n(n+3)}{2} \left(\frac{a+b-2}{2ab}\right) + \frac{n^3+6n^2+11n}{6ab} \\
& + \frac{(h+1)(a-1)(a-5)}{12a} + \frac{(H+1)(b-1)(b-5)}{12b} - \frac{bh(h+1)(a+2)}{4a} - \frac{aH(H+1)(b+2)}{4b}.
\end{align*}
\begin{lemma}
\label{Ref2}
Let $N$ denote the number of nonnegative integer solutions of the equation $ax+by+z+u=n$. Then $$N = S(a,b;h) + S(b,a;H) + \eta_1(a,b,h).$$
\end{lemma}
\begin{proof}
By elementary theory of generating functions, we know that $N$ is equal to the coefficient of $x^n$ in $$ \frac{1}{(1-x^a)(1-x^b)(1-x)^2}. $$ Let $\zeta_m$ denote $e^{{\frac{2\pi i}{m}}}$.  We know that
$$(1-x^a)(1-x^b)(1-x)^2  =  (1-x)^4  \prod_{k = 1}^{a-1}(1-\zeta_a^{-k}x) \prod_{k = 1}^{b-1}(1-\zeta_b^{-k}x). $$ Since $a$ and $b$ are coprime, $1-\zeta_a^{-k}x$ and $ 1-\zeta_b^{-k}x$ are distinct for all values of $k$. Thus, we obtain the partial fraction decomposition
 \begin{equation}
 \label{Partial Fraction}
 \begin{aligned}
  \frac{1}{(1-x^a)(1-x^b)(1-x)^2} & = \frac{d_1}{1-x} + \frac{d_2}{(1-x)^2} +\frac{d_3}{(1-x)^3} + \frac{d_4}{(1-x)^4} \\
  &\quad + \sum_{k=1}^{a-1}\frac{A_k}{1-\zeta_a^{-k} x} + \sum_{k=1}^{b-1}\frac{B_k}{1-\zeta_b^{-k} x}. 
     \end{aligned}
     \end{equation}
   On comparing the coefficients of $x^n$ on both sides of \eqref{Partial Fraction}, we find
   \begin{equation}
\label{Coefficient}
    N =  d_1 + (n + 1) d_2 + \frac{(n + 2)(n + 1)}{2} d_3 + \frac{(n + 3)(n+2)(n + 1)}{6} d_4 +  \sum_{k=1}^{a-1}A_k   \zeta_a^{-nk} + \sum_{k=1}^{b-1}B_k   \zeta_b^{-nk}. 
   \end{equation}
   If we substitute $x=0$ in \eqref{Partial Fraction}, we get 
   \begin{equation}
      \label{x=0}
    1 = d_1 + d_2 + d_3 + d_4 + \sum_{k=1}^{a-1} A_k + \sum_{k=1}^{b-1} B_k.
   \end{equation}
  Upon subtracting \eqref{x=0} from \eqref{Coefficient}, we get
   \begin{equation}
   \label{Almost}
   \begin{aligned}
   N - 1 &=  nd_2+ \frac{n(n+3)}{2} d_3 + \frac{n^3+6n^2+11n}{6} d_4 \\
   &\quad - \sum_{k=1}^{a-1}A_k(1-\zeta_a^{-nk}) - \sum_{k=1}^{b-1}B_k(1-\zeta_b^{-nk}).
   \end{aligned}
    \end{equation}
   The usual procedure for finding coefficients of a partial fraction expansion gives the following equations.
   \begin{align*}
    d_4 &= \frac{1}{ab},\\
    d_3 &= \frac{a+b-2}{2ab}, \\
    d_2 &= \gamma(a,b), \\
    A_k &= \frac{1}{a(1-\zeta_a^{bk})(1-\zeta_a^{k})^2}, \\
    B_k &= \frac{1}{b(1-\zeta_b^{ck})(1-\zeta_b^{k})^2}. \\
       \end{align*}
   Substituting these back into \eqref{Almost}, we have 
   \begin{equation}
   \label{Formula'}
   N =  1 + n\gamma(a,b) + \frac{n(n+3)}{2} \left(\frac{a+b-2}{2ab}\right) + \frac{n^3+6n^2+11n}{6ab} - \left(\frac{S_1}{a} + \frac{S_2}{b} \right),
   \end{equation}
  where $$ S_1 =  \sum_{k=1}^{a-1} \frac{1-\zeta_a^{-nk}}{(1-\zeta_a^{bk})(1-\zeta_a^{k})^2} $$ and $$S_2 = \sum_{k=1}^{b-1} \frac{1-\zeta_b^{-nk}}{(1-\zeta_b^{ak})(1-\zeta_b^{k})^2}.$$ 
    Next, we find $S_1$ and $S_2$. By definition of $n$, we have $n \equiv -b(h+1) $ (mod $a$), so $\zeta_a^{-nk} = \zeta_a^{b(h+1)k}$, and thus,
 \begin{equation}
  \label{sums}
  \begin{aligned}
    S_1 &=  \sum_{k=1}^{a-1} \frac{1-\zeta_a^{b(h+1)k}}{(1-\zeta_a^{bk})(1-\zeta_a^{k})^2} \\
    &= \sum_{k=1}^{a-1} \sum_{j = 0}^{h} \frac{\zeta_a^{jbk}}{(1-\zeta_a^{k})^2}  \\
    &= \sum_{k=1}^{a-1} \sum_{j = 0}^{h} \frac{1}{(1-\zeta_a^{k})^2} - \sum_{k=1}^{a-1} \sum_{j = 0}^{h} \frac{1-\zeta_a^{jbk}}{(1-\zeta_a^{k})^2}.  
    \end{aligned}
    \end{equation}
 Note that for each $1 \leq k \leq (a-1)$, $\frac{1}{1-\zeta_a^{k}}$ satisfies $\left(1-\frac{1}{x}\right)^a = 1$. That is, for each $1 \leq k \leq (a-1)$, $\frac{1}{1-\zeta_a^{k}}$ is a root of the equation $$ax^{a-1} - {a \choose 2} x^{a-2} + {a \choose 3} x^{a-3} - \cdots = 0. $$  From there, it is easy to see that $$ \sum_{k=1}^{a-1} \frac{1}{(1-\zeta_a^{k})^2} = -\frac{(a-1)(a-5)}{12}, $$
    and thus, changing the order of summations yields
  \begin{equation}
  \label{7}
   \sum_{k=1}^{a-1} \sum_{j = 0}^{h} \frac{1}{(1-\zeta_a^{k})^2} = -\frac{(h+1)(a-1)(a-5)}{12}.  
   \end{equation} 
   Moreover, 
   
   \begin{equation}
   \label{8}
    \begin{aligned}
    \sum_{k=1}^{a-1} \sum_{j = 0}^{h} \frac{1-\zeta_a^{jbk}}{(1-\zeta_a^{k})^2} &= \sum_{k=1}^{a-1} \sum_{j = 1}^{h} \frac{1-\zeta_a^{jbk}}{(1-\zeta_a^{k})^2} \\
    &= \sum_{k=1}^{a-1} \sum_{j = 1}^{h} \sum_{l=0}^{bj-1} \frac{\zeta_a^{kl}}{1-\zeta_a^k} \\
    &= \sum_{k=1}^{a-1} \sum_{j = 1}^{h} \sum_{l=0}^{bj-1} \frac{1}{1-\zeta_a^k} - \sum_{k=1}^{a-1} \sum_{j = 1}^{h} \sum_{l=1}^{bj-1} \frac{1-\zeta_a^{kl}}{1-\zeta_a^k} \\
    &=  \sum_{j = 1}^{h} \sum_{l=0}^{bj-1} \sum_{k=1}^{a-1} \frac{1}{1-\zeta_a^k} - \sum_{k=1}^{a-1} \sum_{j = 1}^{h} \sum_{l=1}^{bj-1} \sum_{m=0}^{l-1} \zeta_a^{mk} \\
    &= \frac{(a-1)bh(h+1)}{4}  - \sum_{k=0}^{a-1} \sum_{j = 1}^{h} \sum_{l=1}^{bj-1} \sum_{m=0}^{l-1} \zeta_a^{mk} + \frac{b^2h(h+1)(2h+1)}{12} - \frac{bh(h+1)}{4} \\
     &= \frac{(a-2)bh(h+1)}{4}  - \sum_{k=0}^{a-1} \sum_{j = 1}^{h} \sum_{l=1}^{bj-1} \sum_{m=0}^{l-1} \zeta_a^{mk} + \frac{b^2h(h+1)(2h+1)}{12}.
\end{aligned}
\end{equation}
We know that $ \sum_{k=0}^{a-1} \zeta_a^{mk} \neq 0$ only if $a$ divides $m$, and in that case, the sum is $a$. Therefore, 
\begin{equation}
\label{9}
    \begin{aligned}
\sum_{k=0}^{a-1} \sum_{j = 1}^{h} \sum_{l=1}^{bj-1} \sum_{m=0}^{l-1} \zeta_a^{mk} &= \sum_{j = 1}^{h} \sum_{l=1}^{bj-1} \sum_{m=0}^{l-1} \sum_{k=0}^{a-1} \zeta_a^{mk} \\
&= a \sum_{j = 1}^{h} \sum_{l=1}^{bj-1} \left(\left \lfloor \frac{l-1}{a} \right \rfloor +1 \right) \\
&= a \sum_{j = 1}^{h} \sum_{l=1}^{bj-1} \left \lfloor \frac{l-1}{a} \right \rfloor + \frac{abh(h+1)}{2} - ah.
\end{aligned}  
\end{equation}
Next, note that $\left \lfloor \frac{l-1}{a} \right \rfloor = \left \lfloor \frac{l}{a} \right \rfloor$ unless $a$ divides $l$. Therefore, 
\begin{equation}
\label{10}
\begin{aligned}
\sum_{j = 1}^{h} \sum_{l=1}^{bj-1} \left \lfloor \frac{l-1}{a} \right \rfloor &= \sum_{j = 1}^{h} \sum_{l=1}^{bj-1} \left \lfloor \frac{l}{a} \right \rfloor -  \sum_{j = 1}^{h}  \left \lfloor \frac{bj-1}{a} \right \rfloor \\
&=  \sum_{j = 1}^{h} \sum_{l=1}^{bj-1} \left \lfloor \frac{l}{a} \right \rfloor -  \sum_{j = 1}^{h}  \left \lfloor \frac{bj}{a} \right \rfloor \\
&=  \sum_{j = 1}^{h} \sum_{l=1}^{bj} \left \lfloor \frac{l}{a} \right \rfloor - 2 \sum_{j = 1}^{h}  \left \lfloor \frac{bj}{a} \right \rfloor. \\
\end{aligned}
\end{equation}
Finally, note that for any $1 \leq j \leq h$, 
\begin{equation}
\label{11}
\begin{aligned}
\sum_{l=1}^{bj} \left \lfloor \frac{l}{a} \right \rfloor &= a\left(1+2+ \cdots +\left(\left \lfloor \frac{bj}{a} \right \rfloor - 1\right)\right) + \left \lfloor \frac{bj}{a} \right \rfloor \left(bj - a\left \lfloor \frac{bj}{a} \right \rfloor + 1 \right) \\
&= \left(bj  \left \lfloor \frac{bj}{a} \right \rfloor - \frac{a}{2} \left \lfloor \frac{bj}{a} \right \rfloor^2\right) - \left(\frac{a}{2}-1\right) \left \lfloor \frac{bj}{a} \right \rfloor \\
&= \frac{a}{2} \left \lfloor \frac{bj}{a} \right \rfloor \left(\frac{2bj}{a} - \left \lfloor \frac{bj}{a} \right \rfloor \right) - \left(\frac{a}{2}-1\right) \left \lfloor \frac{bj}{a} \right \rfloor \\
&= \frac{a}{2} \left(\frac{bj}{a} -  \left\{\frac{bj}{a}\right\} \right) \left(\frac{bj}{a} +  \left\{\frac{bj}{a}\right\} \right) - \left(\frac{a}{2}-1\right) \left \lfloor \frac{bj}{a} \right \rfloor \\
&= \frac{a}{2} \left(\left(\frac{bj}{a}\right)^2 -  \left\{\frac{bj}{a}\right\}^2 \right) - \left(\frac{a}{2}-1\right) \left \lfloor \frac{bj}{a} \right \rfloor. 
\end{aligned}
\end{equation}
Therefore, by \eqref{10} and \eqref{11}, 
\begin{equation}
\label{12}
\begin{aligned}
\sum_{j = 1}^{h} \sum_{l=1}^{bj-1} \left \lfloor \frac{l-1}{a} \right \rfloor &= \frac{b^2h(h+1)(2h+1)}{12a} - \frac{a}{2} \left(\sum_{j=1}^h \left\{ \frac{bj}{a} \right\}^2\right) -  \left(\frac{a}{2}+1\right) \sum_{j=1}^h \left \lfloor \frac{bj}{a} \right \rfloor \\
&= \frac{b^2h(h+1)(2h+1)}{12a} - S(a,b;h).
\end{aligned}
\end{equation}
From \eqref{sums}, \eqref{7},  \eqref{8}, \eqref{9} and \eqref{12}, we get that 
\begin{equation}
\label{13}
S_1= -\frac{(h+1)(a-1)(a-5)}{12} + \frac{bh(h+1)(a+2)}{4} -a \sum_{j = 1}^{h}  \left \lfloor \frac{bj}{a} \right \rfloor-ah - aS(a,b;h).
\end{equation}
Symmetrically, we get 
\begin{equation}
\label{14}
S_2= -\frac{(H+1)(b-1)(b-5)}{12} + \frac{aH(H+1)(b+2)}{4} -b \sum_{j = 1}^{H}  \left \lfloor \frac{aj}{b} \right \rfloor-bH - bS(b,a;H).
\end{equation}
The result now follows from \eqref{Formula'}, \eqref{13} and \eqref{14}.
\end{proof}
Using Lemma \ref{Ref1} and Lemma \ref{Ref2}, we get the following reciprocity relation for $S(a,b;h)$. For brevity of notation, define $$ \eta_2(a,b,h):= \frac{(n+1)(n+2)}{2} + \frac{(a-1)(b-1)(2ab-a-b-6n-7)}{12} - \eta_1(a,b,h). $$
\begin{theorem}
\label{Reci}
For given positive coprime integers $a$ and $b$, and a given natural number $h$, $S(a,b;h)$ satisfies the following reciprocity relationship: $$ S(a,b;h) + S(b,a;H) = \eta_2(a,b,h). $$ 
\end{theorem}
Next, we describe our algorithm for calculating $S(a,b;h)$.
\begin{enumerate}
\item Suppose $a>b$. We express $S(a,b;h)$ in terms of $S(b,a;H)$ using Theorem \ref{Reci}.  
\item Suppose $b \geq a$.  Then, $b=aq+r$ for some $q \geq 1$ and $r < a$. Then, it is easy to observe that 
\begin{equation}
\label{Div}
 S(a,b;h) =  S(a,r;h) + \frac{qh(h+1)(a+2)}{4}. 
 \end{equation}
\item We keep repeating Steps $1$ and $2$ until we are done.
\end{enumerate}

\subsection{An example}
\label{Sub1}
Suppose we want to calculate the value of $T_1(8411,2732;1221)$, that is $$ \sum_{i=1}^{1221} \left\{\frac{2732 \hspace{.05cm} i}{8411}\right\}^2. $$ First, we evaluate $S(8411,2732;1221)$ using the above algorithm. Set $a=8411$, $b=2732$ and $h=1221$ 
in Theorem \ref{Reci} to get 
\begin{equation}
\label{Eqn 31}
 S(8411,2732;1221) + S(2732,8411;2335) =  \frac{5521952154451967}{441901}. 
 \end{equation}
  Using \eqref{Div}, we get  
  \begin{equation}
\label{Eqn 32} 
  S(2732,8411;2335) = S(2732,215;2335) + 11184575280.
  \end{equation}
   Using Theorem \ref{Reci}, we get 
 \begin{equation}
\label{Eqn 33}
S(2732,215;2335) + S(215,2732;31) = \frac{43105956866071}{146845}. 
\end{equation}
  Using \eqref{Div}, we get 
   \begin{equation}
  \label{Eqn 34}
  S(215,2732;31) = S(215,152;31) + 645792. 
  \end{equation}
   Using Theorem \ref{Reci}, we get 
   \begin{equation}
    \label{Eqn 35}
    S(215,152;31) + S(152,215;129) = \frac{62027530983}{65360}. 
    \end{equation}
     Using \eqref{Div}, we get 
     \begin{equation}
   \label{Eqn 36}
      S(152,215;129) = S(152,63;129) + 645645.
      \end{equation}
      Using Theorem \ref{Reci}, we get 
      \begin{equation}
   \label{Eqn 37}
       S(152,63;129) + S(63,152;9) = \frac{1719655381}{6384}. 
       \end{equation}
      Using \eqref{Div}, we get 
      \begin{equation}
   \label{Eqn 37}
       S(63,152;9) = S(63,26;9) + 2925. 
       \end{equation}
        Using Theorem \ref{Reci}, we get 
        \begin{equation}
   \label{Eqn 38}
         S(63,26;9) + S(26,63;21) = \frac{9093619}{1092}. 
         \end{equation}
          Using \eqref{Div}, we get 
          \begin{equation}
   \label{Eqn 39}
   S(26,63;21) = S(26,11;21) + 6468.           
  \end{equation}
          Using Theorem \ref{Reci}, we get 
  \begin{equation}
  \label{Eqn 40} 
   S(26,11;21) + S(11,26;1) = \frac{757997}{572}. 
   \end{equation}
           Finally, it is easy to see that 
         \begin{equation}
\label{Eqn 41}  
            S(11,26;1) = \frac{151}{11}. 
            \end{equation}
From \eqref{Eqn 31} to \eqref{Eqn 41}, we get that $$ S(8411,2732;1221) = \frac{658946167630}{647}. $$ That is, 
\begin{equation}
\label{Eqn 50}  
 \left(\frac{8411}{2}\right) \sum_{i=1}^{1221} \left\{\frac{2732 \hspace{.05cm} i}{8411}\right\}^2  +  \left(\frac{8413}{2}\right) \sum_{i=1}^{1221} \left \lfloor \frac{2732 \hspace{.05cm} i}{8411} \right \rfloor=  \frac{658946167630}{647}. 
  \end{equation}
  
The summation $\sum_{i=1}^{1221} \left \lfloor \frac{2732 \hspace{.05cm} i}{8411} \right \rfloor$  can be easily calculated using the algorithm described in \cite[Section 2.3]{Binner}. However, we provide all the details here for the sake of completeness. 

In order to solve the first sum, we apply Theorem \ref{Quotient} to get 
\begin{equation}
\label{Eqn 17}
 \sum_{i=1}^{1221} \left\lfloor \frac{2732  \hspace{.05cm} i}{8411} \right\rfloor = 483516 -  \sum_{i=1}^{396} \left\lfloor \frac{8411  \hspace{.05cm} i}{2732} \right\rfloor.   
\end{equation}
Then, by the division algorithm,
\begin{equation}
\label{Eqn 18}
\begin{aligned}
\sum_{i=1}^{396} \left\lfloor \frac{8411  \hspace{.05cm} i}{2732}  \right\rfloor&= \sum_{i=1}^{396} \left(3i + \left\lfloor \frac{215i}{2732} \right\rfloor \right) \\
&= 235818 + \sum_{i=1}^{396} \left\lfloor \frac{215i}{2732} \right\rfloor.  
\end{aligned}
\end{equation}
Repeated applications of Theorem \ref{Quotient}, followed by the division algorithm, give the following equations. 
\begin{equation}
\label{Eqn 19}
\begin{aligned}
 \sum_{i=1}^{396} \left\lfloor \frac{215i}{2732} \right\rfloor &= 12276 -   \sum_{i=1}^{31} \left\lfloor \frac{2732 \hspace{.05cm} i}{215} \right\rfloor \\
  &=  6324 -  \sum_{i=1}^{31} \left\lfloor \frac{152i}{215} \right\rfloor,
\end{aligned}
\end{equation}
\begin{equation}
\label{Eqn 20}
\begin{aligned}
   \sum_{i=1}^{31} \left\lfloor \frac{152i}{215} \right\rfloor, &= 651-   \sum_{i=1}^{21} \left\lfloor \frac{215i}{152} \right\rfloor  \\
  &= 420 -   \sum_{i=1}^{21} \left\lfloor \frac{63i}{152} \right\rfloor,  
\end{aligned}
\end{equation}
\begin{equation}
\label{Eqn 21}
\begin{aligned}
   \sum_{i=1}^{21} \left\lfloor \frac{63i}{152} \right\rfloor  &=  168 -  \sum_{i=1}^{8}   \left\lfloor \frac{152i}{63} \right\rfloor  \\
   &= 96 -  \sum_{i=1}^{8}   \left\lfloor \frac{26i}{63} \right\rfloor,       
 \end{aligned}
 \end{equation}
 \begin{equation}
 \label{Eqn 22}
 \begin{aligned}
        \sum_{i=1}^{8}   \left\lfloor \frac{26i}{63} \right\rfloor &=  24 -  \sum_{i=1}^{3} \left\lfloor \frac{63i}{26} \right\rfloor \\
        &= 12 -  \sum_{i=1}^{3} \left\lfloor \frac{11i}{26} \right\rfloor,
      \end{aligned}
      \end{equation}
and
      \begin{equation}
      \label{Eqn 23}
      \begin{aligned}
       \sum_{i=1}^{3} \left\lfloor \frac{11i}{26} \right\rfloor  &= 3 - \sum_{i=1}^{1} \left\lfloor \frac{26i}{11} \right\rfloor \\
       &=1.
\end{aligned}
\end{equation}
From \eqref{Eqn 17} to \eqref{Eqn 23}, we get 
 \begin{equation}
 \label{Eqn 100}
  \sum_{i=1}^{1221} \left\lfloor \frac{2732  \hspace{.05cm} i}{8411} \right\rfloor = 241709. 
   \end{equation}
   From \eqref{Eqn 50} and \eqref{Eqn 100}, we get that 
\begin{equation}
\label{Eqn 500} 
T_1(8411,2732;1221) = \sum_{i=1}^{1221} \left\{\frac{2732 \hspace{.05cm} i}{8411}\right\}^2 = \frac{2219247661}{5441917}. 
\end{equation}
Multiplying both sides of this equation by $8411^2$, the above statement is equivalent to $$ \sum_{i=1}^{1221} r_i^2 = 28850219593, $$ where $r_i$ is the remainder when $2732 \hspace{.05cm} i$ is divided by $8411$.

\section{An algorithm for $T_2(a,b;h)$ and $T_3(a,b;h)$}
\label{Sec3}
Recall our notation from Section \ref{Sec1}.
\begin{itemize}
\item $T_1(a,b;h)= \sum_{i=1}^{h} \{\frac{ib}{a}\}^2$.
\item $T_2(a,b;h)= \sum_{i=1}^{h} i \lfloor \frac{ib}{a} \rfloor$.
\item $T_3(a,b;h)= \sum_{i=1}^{h} \lfloor \frac{ib}{a} \rfloor^2$.
\end{itemize}
Note that 
\begin{align*}
 T_3(a,b;h) - T_1(a,b;h) &= \sum_{i=1}^{h} \frac{ib}{a}\left(\left \lfloor  \frac{ib}{a} \right \rfloor -  \left\{\frac{ib}{a}\right\} \right) \\
 &= \sum_{i=1}^{h} \frac{ib}{a}\left(2 \left \lfloor  \frac{ib}{a} \right \rfloor -  \frac{ib}{a} \right) \\
 &= \frac{2b}{a} T_2(a,b;h) - \frac{b^2h(h+1)(2h+1)}{6a^2}.
 \end{align*}
Thus, we get the following relationship between $T_1(a,b;h)$, $T_2(a,b;h)$ and $T_3(a,b;h)$.
\begin{equation}
\label{Imp*}
T_3(a,b;h) =  T_1(a,b;h) + \frac{2b}{a} T_2(a,b;h)- \frac{b^2h(h+1)(2h+1)}{6a^2}.
\end{equation}

\subsection{Reciprocity relation for $T_2(a,b;h)$}

Next, we use another method to calculate $T_3(a,b;h)$. We generalize the ideas in the proof of Theorem \ref{Quotient} described in \cite{Binner}. For the sake of completeness, we provide all the details here. Let $h'$ denote the quantity $\left \lfloor \frac{bh}{a} \right \rfloor$. Then, $$ T_3(a,b;h) = \sum_{t=1}^{h'} t^2n_t, $$  where $n_t$ is the number of $i$ such that $1 \leq i \leq h$ and $\left\lfloor \frac{ib}{a} \right\rfloor = t$. Clearly, if $ t < h'$, then $$n_t = \left \lfloor \frac{(t+1)a}{b} \right \rfloor - \left \lfloor \frac{ta}{b} \right \rfloor; $$  if $t=h'$, then $$n_t  = h - \left \lfloor \frac{h'a}{b} \right \rfloor.$$ Therefore, 
\begin{align*}
\sum_{i=1}^{h}\left \lfloor \frac{ib}{a} \right \rfloor^2 &= \sum_{t=1}^{h'-1} \left(\left\lfloor\frac{(t+1)a}{b}\right\rfloor - \left\lfloor \frac{ta}{b} \right\rfloor \right) t^2 + \left(h-\left\lfloor \frac{h'a}{b}\right\rfloor \right)h'^2 \\
&= \sum_{t=1}^{h'-1} \left( t^2 \left \lfloor \frac{(t+1)a}{b} \right \rfloor - (t-1)^2 \left \lfloor \frac{ta}{b} \right \rfloor \right)  - \sum_{t=1}^{h'-1} (2t-1) \left \lfloor \frac{ta}{b} \right \rfloor + \left(h-\left\lfloor \frac{h'a}{b}\right\rfloor \right)h'^2 \\
&= (h'-1)^2 \left \lfloor \frac{h'a}{b} \right \rfloor - \sum_{t=1}^{h'-1} (2t-1) \left \lfloor \frac{ta}{b} \right \rfloor + hh'^2 -h'^2 \left\lfloor \frac{h'a}{b}\right\rfloor \\
&= hh'^2 -  \sum_{t=1}^{h'} (2t-1) \left \lfloor \frac{ta}{b} \right \rfloor \\
&= hh'^2 -  2T_2(b,a;h') + \sum_{t=1}^{h'} \left \lfloor \frac{ta}{b} \right \rfloor.
\end{align*}

Thus, we obtain the following relation: 
\begin{equation}
\label{Imp}
 T_3(a,b;h) = hh'^2 -  2T_2(b,a;h') + \sum_{t=1}^{h'} \left \lfloor \frac{ta}{b} \right \rfloor. 
 \end{equation}
 
 Using \eqref{Imp*} and \eqref{Imp}, we get the following reciprocity relation for $T_2(a,b;h)$:
 \begin{equation}
 \label{Imp**}
 T_2(a,b;h) + \frac{a}{b} T_2(b,a;h') = \frac{ahh'^2}{2b} + \frac{a}{2b} \left(\sum_{t=1}^{h'} \left \lfloor \frac{ta}{b} \right \rfloor \right)  -\frac{a}{2b}T_1(a,b;h) + \frac{bh(h+1)(2h+1)}{12a}.
 \end{equation}
 
 We describe our algorithm for calculating $T_2(a,b;h)$. The quantity $T_3(a,b;h)$ can then be easily obtained from $T_1(a,b;h)$ and $T_2(a,b;h)$ using \eqref{Imp*}. Our algorithm for  $T_2(a,b;h)$ is as follows:
 
 \begin{enumerate}
\item Suppose $a>b$. We express $T_2(a,b;h)$ in terms of $T_2(b,a;h')$  using \eqref{Imp**}.  Note that the expression involves the terms $T_1(a,b;h)$ and $\sum_{t=1}^{h'} \left \lfloor \frac{ta}{b} \right \rfloor$. The former can be calculated using the algorithm in Section \ref{Sec2} and the latter can be calculated using Theorem \ref{Quotient}, as described in the algorithm in \cite[Section 2.3]{Binner}.
\item Suppose $b \geq a$.  Then, $b=aq+r$ for some $q \geq 1$ and $r < a$. Then, it is easy to observe that 
\begin{equation}
\label{Eqn 200}
 T_2(a,b;h) =  T_2(a,r;h) + \frac{qh(h+1)(2h+1)}{6}. 
 \end{equation}
\item We keep repeating Steps $1$ and $2$ until we are done.
\end{enumerate}

\subsection{An example}
\label{Sub2}
We return to our example $a=8411$, $b=2732$ and $h=1221$. 
Using our algorithm for $T_1(a,b;h)$ in Section \ref{Sec2} and the algorithm for $\sum_{i=1}^{h'} \left \lfloor \frac{ia}{b} \right \rfloor$ in \cite[Section 2.3]{Binner}, we easily obtain 
\begin{align*}
 T_1(8411,2732;1221) &= \frac{2219247661}{5441917}, \\
 \sum_{i=1}^{396} \left \lfloor \frac{8411 \hspace{.05cm}i}{2732} \right \rfloor &= 241807. 
 \end{align*}
Then using \eqref{Imp**}, 
\begin{equation}
\label{Eqn 202}
 T_2(8411,2732;1221) + \frac{8411}{2732} T_2(2732,8411;396) = \frac{1075804292917}{2732}. 
 \end{equation}
From \eqref{Eqn 200}, we get 
\begin{equation}
\label{Eqn 203}
 T_2(2732,8411;396) = T_2(2732,215;396) + 62334558.
\end{equation}
Using our algorithm for $T_1(a,b;h)$ in Section \ref{Sec2} and the algorithm for $\sum_{i=1}^{h'} \left \lfloor \frac{ia}{b} \right \rfloor$ in \cite[Section 2.3]{Binner}, we easily obtain 
\begin{align*}
 T_1(2732,215;396) &= \frac{489539849}{3731912}, \\
 \sum_{i=1}^{31} \left \lfloor \frac{2732 \hspace{.05cm}i}{215} \right \rfloor &= 6287.
\end{align*}
Then using \eqref{Imp**}, 
\begin{equation}
\label{Eqn 205}
 T_2(2732,215;396) + \frac{2732}{215}T_2(215,2732;31) = \frac{704030131}{215}. 
\end{equation}
From \eqref{Eqn 200}, we get 
\begin{equation}
\label{Eqn 206}
 T_2(215,2732;31) = T_2(215,152;31) + 124992. 
\end{equation}
Using our algorithm for $T_1(a,b;h)$ in Section \ref{Sec2} and the algorithm for $\sum_{i=1}^{h'} \left \lfloor \frac{ia}{b} \right \rfloor$ in \cite[Section 2.3]{Binner}, we easily obtain 
\begin{align*}
 T_1(215,152;31) &= \frac{483579}{46225}, \\
  \sum_{i=1}^{21} \left \lfloor \frac{215 \hspace{.05cm}i}{152} \right \rfloor &= 316.
\end{align*}
Then using \eqref{Imp**}, 
\begin{equation}
\label{Eqn 208}
 T_2(215,152;31) + \frac{215}{152} T_2(152,215;21) = \frac{515533}{38}. 
 \end{equation}
From \eqref{Eqn 200}, we get 
 \begin{equation}
\label{Eqn 209}
 T_2(152,215;21) = T_2(152,63;21)+ 3311. 
\end{equation}
 Using our algorithm for $T_1(a,b;h)$ in Section \ref{Sec2} and the algorithm for $\sum_{i=1}^{h'} \left \lfloor \frac{ia}{b} \right \rfloor$ in \cite[Section 2.3]{Binner}, we easily obtain 
 \begin{align*}
 T_1(152,63;21) &= \frac{164511}{23104}, \\
  \sum_{i=1}^{8} \left \lfloor \frac{152 \hspace{.05cm}i}{63} \right \rfloor &= 83.
 \end{align*}
Then using \eqref{Imp**}, 
 \begin{equation}
\label{Eqn 211}
 T_2(152,63;21) + \frac{152}{63} T_2(63,152;8) = \frac{151139}{63}. 
 \end{equation}
From \eqref{Eqn 200}, we get 
 \begin{equation}
\label{Eqn 212}
 T_2(63,152;8) = T_2(63,26;8) + 408. 
 \end{equation}
 Using our algorithm for $T_1(a,b;h)$ in Section \ref{Sec2} and the algorithm for $\sum_{i=1}^{h'} \left \lfloor \frac{ia}{b} \right \rfloor$ in \cite[Section 2.3]{Binner}, we easily obtain 
 \begin{align*}
 T_1(63,26;8) &= \frac{3233}{1323}, \\
 \sum_{i=1}^{3} \left \lfloor \frac{63 \hspace{.05cm}i}{26} \right \rfloor &= 13.
 \end{align*}
Then using \eqref{Imp**}, 
 \begin{equation}
\label{Eqn 214}
 T_2(63,26;8) + \frac{63}{26} T_2(26,63;3) = \frac{3695}{26}. 
\end{equation}
 From \eqref{Eqn 200}, we get 
 \begin{equation}
\label{Eqn 215}
 T_2(26,63;3) = T_2(26,11;3) + 28. 
 \end{equation}
Using our algorithm for $T_1(a,b;h)$ in Section \ref{Sec2} and the algorithm for $\sum_{i=1}^{h'} \left \lfloor \frac{ia}{b} \right \rfloor$ in \cite[Section 2.3]{Binner}, we easily obtain 
 \begin{align*}
 T_1(26,11;3) &= \frac{327}{338}, \\
 \sum_{i=1}^{1} \left \lfloor \frac{26 \hspace{.05cm}i}{11} \right \rfloor &= 2.
 \end{align*}
Then using \eqref{Imp**}, 
 \begin{equation}
\label{Eqn 216}
 T_2(26,11;3) + \frac{26}{11} T_2(11,26;1) = \frac{85}{11}. 
\end{equation}
 From \eqref{Eqn 200}, we get 
 \begin{equation}
\label{Eqn 217}
 T_2(11,26;1) = T_2(11,4;1) + 2. 
 \end{equation}
It is easy to see that 
 \begin{equation}
\label{Eqn 218}
 T_2(11,4;1) = 0.
\end{equation}
 From \eqref{Eqn 202} to \eqref{Eqn 218}, it follows that 
 \begin{equation}
\label{Eqn 220}
 T_2(8411,2732;1221) = \sum_{i=1}^{1221} i \left \lfloor \frac{2732 \hspace{.05cm} i}{8411} \right \rfloor = 196956430. 
 \end{equation}
                         
Finally, we use \eqref{Imp*} to calculate $T_3(8411,2732;1211)$ from the values of $T_1(8411,2732;1211)$ and $T_1(8411,2732;1211)$ obtained in \eqref{Eqn 500} and \eqref{Eqn 220}, respectively. 
\begin{align*}
T_3(8411,2732;1221) &= T_1(8411,2732;1221)  + \frac{5464}{8411} T_2(8411,2732;1221) - \frac{348800520350128}{5441917} \\
&=  \frac{2219247661}{5441917} + \frac{5464}{8411} \times 196956430 - \frac{348800520350128}{5441917} \\
&= 63853169. 
\end{align*}
That is, $$ T_3(8411,2732;1221) = \sum_{i=1}^{1221} \left \lfloor \frac{2732 \hspace{.05cm} i}{8411} \right \rfloor^2  =  63853169. $$

\section{Efficiency of the algorithms}
\label{Sec4}
We compare the reciprocity relation in Theorem \ref{Reci} with that in Theorem \ref{Quotient}. The analysis in \cite[Section 2.5]{Binner} shows that $S(a,b;h)$ can be calculated in $O(\log t)$ steps where $t = \max(a,b)$. The quantity $\sum_{i=1}^{h} \left \lfloor \frac{ib}{a} \right \rfloor$ can also be calculated in $O(\log t)$ steps, as described in  \cite[Section 2.5]{Binner}. Therefore, $T_1(a,b;h) = \sum_{i=1}^{h} \{\frac{ib}{a}\}^2$ can be calculated in $O(\log t)$ steps. 

Consider the reciprocity relation for $T_2(a,b;h)$ in \eqref{Imp**}. Note that this is similar to the ones above except that in each step, we need to calculate $T_1(a,b;h)$ and $\sum_{i=1}^{h'} \left \lfloor \frac{ia}{b} \right \rfloor$, both of which require $O(\log t)$ steps. Thus, in order to calculate $T_2(a,b;h)$, we need to apply the reciprocity relation $O(\log t)$ times and each time, we need to perform $O(\log t)$ steps. Hence, the number of steps required for calculating $T_2(a,b;h) = \sum_{i=1}^{h} i \lfloor \frac{ib}{a} \rfloor$ is $O((\log t)^2)$. 

The quantity $T_3(a,b;h)$ can be obtained from $T_1(a,b;h)$ and $T_2(a,b;h)$ using \eqref{Imp*}. Therefore, the number of steps required for calculating $T_3(a,b;h) = \sum_{i=1}^{h} \lfloor \frac{ib}{a} \rfloor^2$ is also $O((\log t)^2)$. 

\section{Acknowledgements}
 I want to thank the Maths Department at SFU for providing me various awards and fellowships which help me conduct my research.


\begin{thebibliography}{50}

\bibitem{Binner} D. S. Binner, The number of solutions to $ax+by+cz = n$ and its relation to quadratic residues. {\it Journal of Integer Sequences}, {\bf 23} (20.6.5), 2020.
\bibitem{BS} T. C. Brown and P. J. Shiue, A remark related to the Frobenius problem, {\it Fibonacci Quart.}, {\bf 31} , 32--36, 1993.
\bibitem{Sylvester82} J. J. Sylvester, On subinvariants, i.e. semi-invariants to binary quantics of an unlimited order, {\it Amer. J. Math.} {\bf 5}, 79--136, 1882.
\bibitem{Sylvester} J. J. Sylvester, Problem $7382$, \emph{Mathematical Questions, with their Solutions, from the Educational Times} {\bf 41}, 21, 1884.
\bibitem{AT} A. Tripathi, The number of solutions to $ax+by=n$, {\it Fibonacci Quart.} {\bf 38}, 290--293, 2000.

\end{thebibliography}
\end{document}